\documentclass{article}

\usepackage[letterpaper,top=2cm,bottom=2cm,left=3cm,right=3cm,marginparwidth=1.75cm]{geometry}

\usepackage{amsmath}
\usepackage{amssymb}
\usepackage{amsthm}
\usepackage{graphicx}
\usepackage[colorlinks=true, allcolors=blue]{hyperref}
\usepackage{xcolor}

\newtheorem{theorem}{Theorem}[section]
\newtheorem{lemma}[theorem]{Lemma}
\newtheorem{proposition}[theorem]{Proposition}
\newtheorem{definition}[theorem]{Definition}
\newtheorem{corollary}[theorem]{Corollary}
\newtheorem{example}[theorem]{Example}
\newtheorem{system}[theorem]{System}
\numberwithin{equation}{section}
\newtheorem*{remark}{Remark}
\newcommand{\rev}[1]{{#1}}

\title{Stabilizability and lower spectral radius for linear switched systems with singular matrices}
\author{Carl P. Dettmann and Chenmiao Zhang}

\begin{document}
\maketitle
\begin{abstract}
    We investigate the stabilizability of linear discrete-time switched systems with singular matrices, focusing on the spectral radius in this context. 
    A new lower bound of the stabilizability radius is proposed, which is applicable to any matrix set. \rev{Switched systems with rank one singular matrices are discussed: The stabilizability radius and the joint spectral subradius are equal for such systems.
    Detailed analysis of the stabilizability radius of two-dimensional switched systems, consisting of a singular matrix and a matrix with complex eigenvalues or real eigenvalues, are presented. The condition when an infinitely long aperiodic optimal sequence appears of such system is also discussed. 
    Other properties of switched systems with singular matrices are also discussed along with examples.}
\end{abstract}

\section{Introduction}

Switched systems have been widely studied in both theoretical and applied contexts. Comprehensive descriptions of the formation and stability of such systems can be found in \cite{sun2006switched} and \cite{sun2011stability}.
Ref.~\cite{chitour2025switched} contains more recent works on continuous time switched linear systems. From a control and analytical perspective, Ref.~\cite{patel2025globally} consider the control of switched linear systems, the Lyapunov analysis of switched systems considered in \cite{della2024multiple} and \cite{della2022continuous}.
Beyond purely theoretical inquiries, switched systems have also proven effective in modeling real-world phenomena, such as the epidemiological spread framework proposed by \cite{bliman2025framework} and designing of multiphase trajectories for aerospace vehicles in \cite{saranathan2018relaxed}.

Despite their structurally simple form, determining the stability and stabilizability of these systems remains highly challenging, even in the linear case.
The discrete-time linear switched system of the given matrix set $\mathcal{M} \subseteq \mathbb{R}^{n \times n}, \; n \geq 2$ is given by
\begin{equation}\label{eq5}
    x(t+1) = M_{\sigma(t)}x(t), \quad x(0)=x_0,
\end{equation}
where $M_{\sigma(t)} \in \mathcal{M}$, $\sigma(t) \in \{1, \dots, m\}$ and $m$ is the cardinality of $\mathcal{M}$. \rev{Notice that $m$ may be infinite. However, by Corollary $3.4$ and Theorem $3.5$ in \cite{stanford1994some}, it is sufficient to only consider finite $m$ in this paper.}
There are many general results on the stability and stabilizability of such system (for example, see \cite{sun2011stability}). However, we are not aware of any study focusing on the case where the matrix set $\mathcal{M}$ contains singular matrices.
In the literature, the term ``singular switched system'' refers to a special kind of switched system with many known results (see \cite{zhou2013stability} and \cite{lewis1986survey}), which is different from the one we consider in this paper.  

To describe the stability and stabilizability of system (\ref{eq5}), the joint spectral properties are widely used. 
Let $\mathcal{M}^T$ denote the set that contains all $M_{\sigma(t)} \cdots M_{\sigma(1)}$. Then the joint spectral radius of system (\ref{eq5}) is defined as
\begin{equation}
    \hat{\rho}(\mathcal{M}) = \lim_{T \rightarrow \infty} \sup_{A \in \mathcal{M}^T} \|A\|^{1/T},
\end{equation}
for any norm. Then the joint spectral subradius is defined as
\begin{equation}
    \check{\rho}(\mathcal{M}) = \lim_{T \rightarrow \infty} \inf_{A \in \mathcal{M}^T} \|A\|^{1/T},
\end{equation}
for any norm.
A new radius proposed in \cite{jungers2017feedback}, called the stabilizability radius, is defined as follows. 
\begin{equation}\label{eq13}
    \tilde{\rho} = \sup_{x_0 \in \mathbb{R}^n} \inf \{ \lambda \geq 0 \, | \, \exists \sigma(\cdot), c > 0 \; s.t. \; \|x(t)\| \leq c \lambda^t \|x_0\| \, \forall t \geq 0\},
\end{equation}
These three radii trivially satisfy $\tilde{\rho}\leq\check{\rho}\leq\hat{\rho}$ and describe different properties of the switched system with different terminology in the literature. In this paper, we have
\begin{itemize}
    \item A switched system is called asymptotically stable if and only if $\hat{\rho}(\mathcal{M}) < 1$ (see, for example, \cite{guglielmi2013stability}).
    \item A switched system is called uniformly stabilizable or open loop stabilizable in the literature if $\check{\rho}(\mathcal{M}) < 1$ (see, for example, \cite{jungers2009joint}).
    \item A switched system is called pointwise stabilizable if $\tilde{\rho}(\mathcal{M}) < 1$ (proposition 2.5 in \cite{jungers2017feedback}).
\end{itemize}
Stanford first showed that pointwise stabilizability is different from uniform stabilizability in \cite{stanford1979stability}.

The stabilizability radius $\tilde{\rho}$, first introduced in \cite{jungers2017feedback}, had many implicit studies beforehand, for example, Ref.~\cite{lee2010supervisory} considered this problem in measurement scheduling and other cases,
and Ref.~\cite{hernandez2011discrete} considered this problem with applications in HIV drug control.
For non-singular matrix sets, $\tilde{\rho}$ has been proved to have several lower bounds in \cite{dettmann2020lower},
\rev{but it would be good to know a lower bound of $\tilde{\rho}$ for an arbitrary matrix set. Let us mention two simple examples
\begin{example}\label{ex5} 
    (Based on an example first given in \cite{stanford1979stability}) Switched system with $\mathcal{M}=\{M_1,M_2\}$ such that
\begin{equation}
    M_1=\left(\begin{matrix}
        \frac{1}{2} &0 \\
        0 &2
    \end{matrix}\right), \quad
    M_2=\left(\begin{matrix}
        \cos \frac{\pi}{6} &-\sin \frac{\pi}{6}\\
        \sin \frac{\pi}{6} & \cos \frac{\pi}{6}
    \end{matrix}\right),
\end{equation}
\end{example} 
\begin{example}\label{ex6}
    Switched system with $\mathcal{M}=\{M_1, M_2\}$ such that
\begin{equation}
    M_1=\left(\begin{matrix}
        \frac{1}{2} &0 &0 \\
        0 &2 &0 \\
        0 &0 &0
    \end{matrix}\right), \quad
    M_2=\left(\begin{matrix}
        \cos \frac{\pi}{6} & -\sin \frac{\pi}{6} & 0\\
        \sin \frac{\pi}{6} & \cos \frac{\pi}{6} & 0\\
        0 & 0 & 1
    \end{matrix}\right).
\end{equation}
\end{example}
This Example \ref{ex5} has been well studied in \cite{dettmann2020lower}; the lower bound given by Theorem $2$ in \cite{dettmann2020lower} gives us $\tilde{\rho} \geq \sqrt{2}/2$ for Example \ref{ex5}.
In Example \ref{ex6}, $M_1$ projects any three dimensional vector onto the $(x_1,x_2)$ plane, thus Example \ref{ex5} and \ref{ex6} should have the same dynamics.
However, as far as we know, none of the existing lower bound works for Example \ref{ex6}, because $M_1$ is a singular matrix. Also in applications mentioned above, the matrix set could contain singular matrices. }

\rev{This leads us to two natural questions: 
\begin{itemize}
    \item Is there a more general lower bound that remains valid in the presence of singular matrices?
    \item What is the dynamics of linear switched systems with singular matrices, and in particular, how does singularity affect their stabilizability?
\end{itemize}  
By addressing these two questions, we aim to gain a deeper understanding of linear switched systems with singular matrices.} 

To the best of our knowledge, there are no previous works focusing on switched system with singular matrices. However, a related setting, the Alternating Projection Method (APM), has many known results (see for example \cite{escalante2011alternating}). In $\mathbb{R}^n$, the setting of APM is related to switched systems with orthogonal projection matrices. A matrix $M$ is an orthogonal projection if it satisfies $M^2=M=M^T$. 

 The question whether an initial point converges to a closed subspace is often considered in the setting of APM. If we take the singleton with origin as the closed subspace, then this is related to the stabilizability problem above. Let us mention \cite{kopecka2020products} for considering the singleton with the origin as the intersection of subspaces of an infinite dimensional space. We have the following proposition to transform an important theorem in the study of APM into a stabilizability result:
\begin{proposition}\label{pro1}
    Let $M_1, \dots, M_{k},\; k \leq m$ be orthogonal projection matrices in $\mathcal{M}$. If
    \begin{equation}
        \bigcap_{i=1}^{k} \mathbf{Im}(M_i) = \{(0, \dots,0)^T\},
    \end{equation}
    then $\mathcal{M}$ is stabilizable, where $\mathbf{Im}(\cdot)$ denotes the image of matrices.
\end{proposition}
The proof of Proposition $\ref{pro1}$ comes immediately from Theorem 3.4 in \cite{escalante2011alternating}. 

In this paper, we establish a new lower bound of the stabilizability radius for any finite matrix set. 
In the meantime, we do a detailed analysis of $\tilde{\rho}$ of a two dimensional switched system with singular matrices, and show how to compute the stability radius $\tilde{\rho}$ using continued fractions and Diophantine approximation.
\rev{We also investigate the finiteness property (formally defined in Definition \ref{def1}), i.e. having finite or periodic optimal sequence, of these two dimensional cases.} Other properties of systems are mentioned as they appear.

Some notations are used in this paper with their well-known meanings. The area denotes the $n-1$ dimensional Lebesgue measure of a set in $\mathbb{R}^n$. For two functions $f$ and $g$, we write $f(x) \sim g(x)$ as $x \rightarrow 0$ if $\lim_{x \rightarrow 0} f(x)/g(x) = 1$. 
The notation $\| \cdot \|$ will now specifically denote the Euclidean norm in $\mathbb{R}^n$. \rev{The notation $\| \cdot \|_N$ from number theory denotes the distance to the closed integer}. \rev{$\mathbb{N}$ denotes natural numbers including zero}. 

In section $2$, we show the $\tilde{\rho}$ is lower bounded by $\check{\rho}$ divided by the cardinality of $\mathcal{M}$.
Section $3$ is dedicated to the switched systems with rank one singular matrices, The lower spectral radius equals stabilizability radius for such systems. Other properties of such systems are discussed as well.
\rev{In section $4$ and $5$, we discuss the stabilizability radius of the two dimensional switched system with two matrices and one of them is a singular matrix.
Section $4$ is about switched system with a singular matrix and a rotation matrix, i.e. with complex eigenvalues. We also show the parameter set, such that $\tilde{\rho} = c, 0 \leq c < 1$ and the optimal sequence is infinitely long and aperiodic, has zero Hausdorff dimension in this section.
Section $5$ is about switched system with a singular matrix and a hyperbolic matrix, i.e. with real eigenvalues. The condition when optimal sequence is infinitely long and aperiodic of these systems is also discussed.
Section $6$ contains several examples of switched systems with singular matrices about the relationship between the smallest singular value and $\tilde{\rho}$. One strategy of finding the optimal switching law is also mentioned.}

\section{Lower Bound for the stabilizability radius}
This section pertains to arbitrary sets of matrices, singular or non-singular.
Let $A_{\sigma,t}$ denote the product $M_{\sigma(t)} \cdots M_{\sigma(1)}$.
Suppose $s^2_1(A_{\sigma,t}), \dots, s^2_n(A_{\sigma,t})$ are ordered eigenvalues of $A^T_{\sigma,t}A_{\sigma,t}$ and $v_1, \dots, v_n$ are corresponding orthonormal eigenvectors, where $s_1$ is smallest singular value of $A_{\sigma,t}$. 
We have two properties of $v_i$.
\rev{\begin{lemma}
    The $v_i$ for $i=1, \dots, n$, called singular vectors, satisfies
    \begin{enumerate}
        \item $A_{\sigma,t}v_i$ are mutually orthogonal. 
        \item $\|A_{\sigma,t} v_i\| =s_i(A_{\sigma,t})$. 
    \end{enumerate}
    for $i=1, \dots, n$.
\end{lemma}
\begin{proof}
    For the first property, we have 
    \begin{equation}
    \left<A_{\sigma,t}v_i,A_{\sigma,t}v_j\right>=\left<v_i,A^T_{\sigma,t}A_{\sigma,t}v_j\right>=s^2_j \left<v_i,v_j\right>.
    \end{equation}
    Then $A_{\sigma,t}v_i$ are orthogonal because $v_i$ are orthogonal, for $i=1, \dots, n$.
    For the second property, we have
    \begin{equation}\label{eq23}
    \|A_{\sigma,t}v_i\|^2 = \left<A_{\sigma,t}v_i,A_{\sigma,t},v_i\right>=\left<v_i,A^T_{\sigma,t}A_{\sigma,t}v_i\right>=s^2_i.
    \end{equation}
    The second property follows from taking square root of both sides of \eqref{eq23}.
\end{proof}}
\rev{Let $B_r$ denote the ball with radius $r$ centered at the origin and $S^{n-1}$ denote the $n$ dimensional unit sphere.
For $x_0 \in S^{n-1}$, we have $x(t) \in B_r$ for some $r$. Now we want a cover of the intersection of the pre-image of $A_{\sigma,t}$ and $S^{n-1}$, i.e. $S_{r,A_{\sigma,t}} = \{x \in S^{n-1} | \|A_{\sigma,t} x\| \leq r\}$.}
Suppose $s_1, \dots, s_\kappa$ are zero singular values, then $v_1, \dots, v_\kappa$ form an orthonormal basis of $\mathbf{Ker}(A_{\sigma,t})$.
Let $(x_1, \dots, x_n)$ denote the coordinate of $x \in \mathbb{R}^n$ in terms of $v_1, \dots, v_n$.
\rev{We can characterize $S_{r,A_\sigma}$ as follows:
\begin{lemma}\label{lm10}
    For any nonzero $A_{\sigma,t}$, let $S' = \{x \in S^{n-1} | |x_n| \leq r/s_n(A_{\sigma,t}) \}$, then $S_{r,A_{\sigma,t}} \subseteq S'$
\end{lemma}
\begin{proof}
    Since $A_{\sigma,t} v_j = 0$ for $j=1, \dots,\kappa$, for any $x \in S_{r,A_{\sigma,t}}$, we have
    \begin{equation}
        \|A_{\sigma,t} x\|^2 = \|\sum_{i=1}^{n} x_i A_{\sigma,t} v_j \|^2 = \sum_{i=\kappa+1}^{n} s^2_i(A_{\sigma,t}) x_i^2 \leq r^2.
    \end{equation}
    Notice that
    \begin{equation}
        \sum_{i=\kappa+1}^{n} s^2_i(A_{\sigma,t}) x_i^2 \geq s^2_n(A_{\sigma,t}) x_n^2.
    \end{equation}
    We have $S_{A_{\sigma,t},r} \subseteq S'$.
\end{proof}}
This $S'$ in Lemma \ref{lm10} is a spherical segment.  
The area of $S'$ can be obtained using the area of spherical caps.
A formula for area of caps is given by Lemma \ref{lm4}.
\begin{lemma}\label{lm4}
    ((1) in \cite{li2010concise})The surface area of the cap of the unit sphere $S^{n-1}$, i.e. $T_n(h)=\{x \in S^{n-1}|x_1 \geq h\}$ with $x_1$ be one of coordinates of $x$ \rev{and $h>0$}, is 
    \begin{equation}
        |T_n(h)| = \frac{1}{2} |S^{n-1}| I\left(1-h^2;\frac{n-1}{2},\frac{1}{2}\right),
    \end{equation}
    where $I(h;\cdot, \cdot)$ denotes the regularized incomplete beta function, defined as
    \begin{equation}
        I(h;a,b) = \frac{\int_0^{h} t^{a-1}(1-t)^{b-1}\text{d}t}{\int_0^1 t^{a-1}(1-t)^{b-1}\text{d}t}, \quad 0 < h < 1.
    \end{equation}
\end{lemma}
A useful equality of $I(h;\cdot, \cdot)$ is given by
\begin{equation}
    I\left(h;\frac{n-1}{2},\frac{1}{2}\right) = 1-I\left(1-h;\frac{1}{2},\frac{n-1}{2}\right).
\end{equation}
One can prove this by changing variables in the integral definition of $I(h;\cdot, \cdot)$. \rev{Now by Lemmas \ref{lm10} and \ref{lm4}, we can generate a upper bound of $|S_{A_{\sigma,t},r}|$}.
\begin{lemma}
    The area of $S_{r,A_{\sigma,t}}$ is upper bounded by 
    \begin{equation}
        |S_{r,A_{\sigma,t}}| \leq |S^{n-1}|I\left(\frac{r^2}{s^2_n\left(A_{\sigma,t}\right)}; \frac{1}{2}, \frac{n-1}{2}\right).
    \end{equation}
    For any $t \in \mathbb{N}^+$.
\end{lemma}
\begin{proof}
The area of $S'$ is given by
\begin{equation}
    |S'| = |S^{n-1}| - 2 \left|T_n\left(\frac{r}{s_n\left(A_{\sigma,t}\right)}\right)\right|.
\end{equation}
By Lemma \ref{lm10}, we have
\begin{equation}
    |S_{r,A_{\sigma,t}}| \leq |S^{n-1}|\left(1- I\left(1-\frac{r^2}{s^2_n\left(A_{\sigma,t}\right)}; \frac{n-1}{2}, \frac{1}{2}\right)\right) 
                = |S^{n-1}|I\left(\frac{r^2}{s^2_n\left(A_{\sigma,t}\right)}; \frac{1}{2}, \frac{n-1}{2}\right).
\end{equation} 
\end{proof}
\rev{This upper bound of the area of $S_{r,A_{\sigma,t}}$ leads to a lower bound of stabilizability radius, but we need one more lemma to prove it.
\begin{lemma}\label{lm11}
    For an integer $T \geq 1$, let $S(T)$ denotes the smallest $s_n(A_{\sigma,T})$ for all $\sigma(\cdot)$. We have
    \begin{equation}\label{eq9}
        \liminf_{T \rightarrow \infty} S(T)^{1/T} = \liminf_{T \rightarrow \infty} \{s_n(A_{\sigma,t})^{1/t} \;|\; \sigma(t) \in \{1,\dots,m\}, t \in\{1,\dots,T\}\},
    \end{equation}
    i.e.
    \begin{equation}
        \liminf_{T \rightarrow \infty} S(T)^{1/T} = \check{\rho}.
    \end{equation}
\end{lemma}
\begin{proof}
    Clearly, $S(T) \in \{s_n(A_{\sigma,t})^{1/t} \;|\; \sigma(t) \in \{1,\dots,m\}, t \in\{1,\dots,T\}\}$. Then the LHS of \eqref{eq9} is greater or equals to the RHS of \eqref{eq9}.
    Now Suppose the LHS of \eqref{eq9} is larger than the RHS of \eqref{eq9}, then there exist a $T'>0$ and $A_{\sigma,T'}$ such that 
    \begin{equation}
        s_n(A_{\sigma,T'}) < \liminf_{T \rightarrow \infty} S(T)^{1/T}.
    \end{equation}
    This is impossible as $S(T') \leq s_n(A_{\sigma,T'})$. 
    Then the LHS of \eqref{eq9} must equal to the RHS of \eqref{eq9} and The RHS of \eqref{eq9} is the $\check{\rho}$ of the matrix set $\mathcal{M}$ (see, for example, Theorem 1.1 in \cite{jungers2009joint})
\end{proof}}
\begin{theorem}\label{thm1}
    The stabilizability radius satisfies \rev{the following inequality}.
    \begin{equation}
        \frac{\check{\rho}}{m} \leq \tilde{\rho} \leq \check{\rho}.
    \end{equation}
\end{theorem}
\begin{proof}
    The inequality $\tilde{\rho} \leq \check{\rho}$ is from the definition of $\tilde{\rho}$.  
    For a given $T$, the area of the union of $S_{r,A_{\sigma,T}}$ under all $\sigma(\cdot)$ has an upper bound given by
\begin{equation}
    \left|\bigcup_\sigma S_{r,A_{\sigma,T}}\right| \leq \sum_{\sigma} |S_{r,A_{\sigma,T}}| \leq \sum_{\sigma} |S^{n-1}|I\left(\frac{r^2}{s^2_n(A_{\sigma,T})}; \frac{1}{2}, \frac{n-1}{2}\right)
\end{equation}
The sum above is over all possible choices of $\sigma(\cdot)$.
Since $I(h,1/2,(n-1)/2)$ is the cumulative distribution function of the beta distribution, it monotonically increases for $ 0 \leq h \leq 1$. 
Then
\begin{equation}
    I\left(\frac{r^2}{s^2_n(A_{\sigma,T})};\frac{1}{2}, \frac{n-1}{2}\right) \leq I \left(\frac{r^2}{S(T)^2};\frac{1}{2}, \frac{n-1}{2}\right)
\end{equation} 
holds for any $\sigma(\cdot)$, which means that 
\begin{equation}
    \left|\bigcup_\sigma S_{r,A_{\sigma,T}}\right| \leq m^T |S^{n-1}|I\left(\frac{r^{2}}{S(T)^2}; \frac{1}{2}, \frac{n-1}{2}\right).
\end{equation}
From the definition of $\tilde{\rho}$, for any $\lambda > \tilde{\rho}$, 
there exists a positive constant $c$ such that the union of pre-images, $S_{c\lambda^T,A_{\sigma,T}}$ for all $\sigma(\cdot)$, covers the unit sphere. Hence that
\begin{equation}\label{eq2}
    |S^{n-1}| = \left|\bigcup_\sigma S_{c\lambda^T,A_{\sigma,T}}\right| \leq m^T |S^{n-1}|I\left(\frac{c^2\lambda^{2T}}{S(T)^2}; \frac{1}{2}, \frac{n-1}{2}\right).
\end{equation}
Rearranging (\ref{eq2}), we have a bound of $\lambda$ given by
\begin{equation}
    I\left(\frac{c^2\lambda^{2T}}{S(T)^2}; \frac{1}{2}, \frac{n-1}{2}\right) \geq \frac{1}{m^{T}}.
\end{equation}

\rev{Now we need the asymptotic behavior of $I(h;1/2,(n-1)/2)$ when $h$ tends to $0$ to get a lower bound of $\tilde{\rho}$.
Notice that 
\begin{equation}
    I(h;1/2,(n-1)/2) = C\int_0^{h} t^{-\frac{1}{2}}(1-t)^{\frac{n-3}{2}}\text{d}t,
\end{equation}
for a constant $C$. Let $g(x)$ denote the integrand, then for any $n \geq 2$, as $x \rightarrow 0$, we have
$g(x) \sim x^{-\frac{1}{2}}$, i.e.
\begin{equation}
    \int_{0}^{h} g(x) \text{d}x \sim 2\sqrt{h}. 
\end{equation}
Then for large enough $T$, we have
\begin{equation}
    (2+\epsilon) C\frac{c\lambda^{T}}{S(T)} \geq \frac{1}{m^T},
\end{equation}
\rev{for any $\epsilon>0$}, which is
\begin{equation}
    ((2+\epsilon)Cc)^{1/T}\lambda \geq \frac{S(T)^{1/T}}{m}.
\end{equation}
Let $T$ tend to infinity. By lemma \ref{lm11}, we have
\begin{equation}
    \lambda \geq \frac{\check{\rho}}{m}
\end{equation}
for any $\lambda >\tilde{\rho}$. Our lower bound is obtained by letting $\lambda$ tend to $\tilde{\rho}$.}
\end{proof}
\rev{In general this is not a sharp bound, especially for large $m$. 
For Example \ref{ex5}, we have $\check{\rho}=1$. From Theorem \ref{thm1}, we know that $\tilde{\rho} \geq 1/2$, which is not as sharp as lower bound given by Theorem $2$ in \cite{dettmann2020lower} mentioned above.
However Theorem \ref{thm1} works for Example \ref{ex6}. As mentioned above, Example \ref{ex6} has the same dynamics as Example \ref{ex5}. Then they have the same joint spectral subradius $\check{\rho}=1$, which gives the same lower bound of $\tilde{\rho}$ as expected.}

From Theorem \ref{thm1}, \rev{we know that the stabilizability radius satisfies $\tilde{\rho} \geq \check{\rho}/m$}.
An immediate corollary of Theorem \ref{thm1} is that when $\tilde{\rho}=0$, we have $\check{\rho}=0$. 
In other word, when a matrix set pointwise converges to zero quickly enough, it converges to zero uniformly.

\section{Rank one singular matrices}
\rev{The rank one singular matrix (in $\mathbb{R}^{n \times n}$), i.e. the singular matrix with one dimensional image, is important in the study of the stabilizability radius. 
Due to the one dimensional image of such a matrix, the difference between $\tilde{\rho}$ and $\check{\rho}$ is eliminated.}
\begin{proposition}\label{lm7}
    \rev{If $\mathcal{M}$ contains at least one rank one singular matrix, then $\tilde{\rho}(\mathcal{M}) = \check{\rho}(\mathcal{M})$}
\end{proposition}
\begin{proof}
    \rev{Let us recall that the definition of $\tilde{\rho}$ of a matrix set is given by
    \begin{equation}\label{eq12}
        \tilde{\rho}(\mathcal{M}) = \sup_{x_0 \in S^{n-1}} \inf \{ \lambda \geq 0 \, | \, \exists \sigma(\cdot), c > 0 \; s.t. \; \|x(t)\| \leq c \lambda^t \, \forall t \geq 0\},
    \end{equation}
    The definition \eqref{eq12} is slightly different but equivalent to the definition \eqref{eq13}.
    Let $\tilde{\rho}_{x_0}(\mathcal{M})$ denote the infimum in \eqref{eq12}, i.e.
    \begin{equation}
        \tilde{\rho}(\mathcal{M})=\sup_{x_0 \in S^{n-1}} \tilde{\rho}_{x_0}(\mathcal{M})
    \end{equation}
    Let $\mathcal{M}_0$ denote a $\tilde{\rho}$-minimal matrix set containing at least one singular matrix with a one-dimensional image. Let $m_0$ denote the cardinality of $\mathcal{M}_0$, where each matrix is indexed from $\{1, \dots, m_0\}$. 
    Let $M_1$ be a singular matrix in $\mathcal{M}_0$ with one dimensional image, and $v_1$ be a normalized vector in its image, For any $x_0$ we have
    \begin{equation}
        \|M_1 x_0\| = C_{x_0}\|v_1\|,
    \end{equation}
    for some constant depending on $x_0$. Then from the definition of $\tilde{\rho}_{x_0}(\mathcal{M})$, we have
    \begin{equation}
        \tilde{\rho}_{x_0}(\mathcal{M}) \leq \tilde{\rho}_{M_1x_0}(\mathcal{M}) = \tilde{\rho}_{v_1}(\mathcal{M}), 
    \end{equation}
    for any $x_0 \in S^{n-1} \setminus \text{Ker}(M_1)$. 
    Then $\tilde{\rho}(\mathcal{M})=\tilde{\rho}_{v_1}(\mathcal{M})$, which is independent from initial vectors.}  
\end{proof}
\begin{corollary}\label{lm16}
    We have $\tilde{\rho}=\check{\rho}$ for any two dimensional matrix set with at least one singular matrix.
\end{corollary}
However, for arbitrary matrix set with singular matrices, we do not have such properties. In Example \ref{ex6}, we have $\check{\rho} = 1$ but this system is pointwise stabilizable as it behaves the same as Example \ref{ex5} (proved in \cite{stanford1979stability}).

Notice that for a given $\mathcal{M}$, \rev{a subset of $\mathcal{M}$ may be sufficient to stabilize the system, leaving some matrices redundant for the purposes of stabilization.} 
To avoid such phenomena, we have following definition.
\begin{definition}
    In this paper, a matrix set $\mathcal{M}$ is called \rev{$\tilde{\rho}$-minimal} if for any \rev{proper subset $\mathcal{M}_1$, i.e.} $\mathcal{M}_1 \subset \mathcal{M}$, we have 
    \begin{equation}
        \tilde{\rho}(\mathcal{M}_1) > \tilde{\rho}(\mathcal{M}).
    \end{equation}
\end{definition}
Not all linear discrete time switching systems are $\tilde{\rho}$-minimal. However, it is sufficient to only consider the stabilizability radius of $\tilde{\rho}$-minimal matrix sets. 
\begin{proposition}\label{lm12} 
    For any $\mathcal{M}$, there exists a $\tilde{\rho}$-minimal subset $\mathcal{M}'$ such that $\tilde{\rho}(\mathcal{M}')=\tilde{\rho}(\mathcal{M})$. 
\end{proposition}
\rev{The proof of Proposition \ref{lm12} trivially comes from Corollary 3.4 in \cite{stanford1994some}.
We have the following proposition of $\mathcal{M}$ with only rank one singular matrices
\begin{proposition}\label{lm17}
    For any $n \geq 2$ and $m \in \mathbb{Z}^{+}$, we can find at least one matrix set $\mathcal{M} \subseteq \mathbb{R}^{n \times n}$ with cardinality $m$ such that $\mathcal{M}$ is $\tilde{\rho}$-minimal and contains only rank one singular matrices.
\end{proposition}}
\begin{proof}
    \rev{We introduce a $\mathcal{M}$ for $n=2$ and any $m$ in the following. Let $M_i$ denote matrices in $\mathcal{M}$, we have 
    \begin{equation}
        M_i=s_n(M_i)v_i w^T_i, \; i \in {1, \dots, m},
    \end{equation}
    where $s_n(M_i)$ denotes the largest singular value as in section $2$, $v_i$ and $w_i$ are normalized vectors. 
    Now let $s_n(M_i)=1$ and 
    $$w_i=(\cos (\theta_i + \epsilon), \sin (\theta_i + \epsilon))^T,\; v_i=(\cos (\theta_{i-1}+\pi/2), \sin (\theta_{i-1}+\pi/2))^T, \; i \in {1, \dots, m},$$ 
    where $\theta_i \in (0,2\pi)$, for $i \in \{0, \dots, m\}$ and $\epsilon>0$.
    Then we have 
    \begin{equation}
        M_iM_j = \cos( \theta_i -\theta_{j-1} - \pi/2 + \epsilon) v_{i}w^T_j, \; i, \; j \in \{1, \dots, m\}. 
    \end{equation}
    Choosing $\theta_0 =\theta_m = 3\pi/2$ and all $\theta_i$ such that:
    \begin{equation}\label{eq21}
        \theta_i \neq \theta_j +k \pi, \;i \neq j \in \{0, \dots, m-1\}, \; k\in \mathbb{Z}.
    \end{equation}
    Then there exists $\epsilon>0$ such that
    \begin{equation}\label{eq20}
        |\sin(\theta_i - \theta_j +\epsilon)| \geq \sin \epsilon, \; i,j \in \{0, \dots, m-1\}.
    \end{equation} 
    Consider any combination of matrices with length $L$. We have
    \begin{equation}
        A_{\sigma, L} =M_{\sigma(L)} \cdots M_{\sigma(1)}= v_{\sigma(L)}w^T_{\sigma(1)} \prod_{i=1}^{L-1} \cos\left( \theta_{\sigma(i)} -\theta_{\sigma(i+1)-1} - \pi/2 + \epsilon\right).   
    \end{equation}
    Then the lower spectral radius of $\mathcal{M}$ is given by
    \begin{equation}\label{eq22}
        \check{\rho} = \lim_{L \rightarrow \infty} \inf_{\sigma}  \|A_{\sigma, L}\|^{1/L} = \lim_{L \rightarrow \infty} \inf_{\sigma} \|v_{\sigma(L)}w^T_{\sigma(1)} \|^{1/L} \left| \prod_{i=1}^{L-1} \sin\left( \theta_{\sigma(i)} -\theta_{\sigma(i+1)-1}+ \epsilon\right)\right|^{1/L} 
    \end{equation}
    By \eqref{eq20} we have
    \begin{equation}
        \check{\rho} \geq \lim_{L \rightarrow \infty} \|v_{\sigma(L)}w^T_{\sigma(1)} \|^{1/L} (\sin \epsilon)^{(L-1)/L} \geq \sin \epsilon.
    \end{equation}
    Choose $\sigma_k=m+1-k$. We have 
    \begin{equation}
        (M_1M_2 \cdots M_m)^k = (\sin \epsilon)^{mk-1} v_1 w^T_m, \; k \in \mathbb{Z}^{+},
    \end{equation}
    is an optimal sequence of $\check{\rho}$. In view of \eqref{eq21}, equality in \eqref{eq20} hold only if $i=j$. Thus the only optimal sequence has $\sigma(i)=\sigma(i+1)-1$ in \eqref{eq22}, that is, $(M_1M_2 \cdots M_m)^k$. 
    By Proposition \ref{lm7}, the optimal sequence of $\tilde{\rho}$ is the same as the optimal sequence of $\check{\rho}$, so this $\mathcal{M}$ is a $\tilde{\rho}$-minimal set.}    
\end{proof}
\rev{Let us mention $\mathcal{M}$ with $n=2$ and $m=3$ as an example of the proof of proposition \ref{lm17}. Let $v_1=(1,0)^T$, $\theta_1=\pi/6$, $\theta_2=5\pi/6$, $w_3 = (\cos (3\pi/2 + \epsilon), \sin (3\pi/2 + \epsilon))^T$.
It is easy to verify that the sequence $M_i^k$ give us 
$$\lim_{L \rightarrow \infty} \|A_{\sigma,L}\|^{1/L} = |\cos(\pi/6 + \epsilon)|,$$
and sequence $(M_iM_j)^k$ give us
$$\lim_{L \rightarrow \infty} \|A_{\sigma,L}\|^{1/L} = |\sin \epsilon \cos(\pi/6 - \epsilon)|^{1/2},$$
for $i,j \in \{1,2,3\}$. Finally, the sequence $(M_1M_2M_3)^k$ give us 
$$\lim_{L \rightarrow \infty} \|A_{\sigma,L}\|^{1/L} = |\sin \epsilon|.$$
Thus for sufficiently small $\epsilon$, $\mathcal{M}$ is a $\tilde{\rho}$-minimal set.}

\section{Two dimensions: Elliptic case}
\subsection{A singular and a rotation matrix}
\rev{In Corollary \ref{lm16}, we mentioned $\tilde{\rho}= \check{\rho}$ for any two switched system with singular matrices. Then to better understand the structure of linear switched systems with singular matrices,} we start from two dimensional linear switched systems.
\begin{system}\label{sys1}
Consider a two dimensional switched system with a set of matrices $\mathcal{M}=\{M_1, M_2\}$ with $M_1$ a singular matrix and $M_2$ a matrix with two complex eigenvalues, namely
\begin{equation}\label{eq4}
    M_1=\left(\begin{matrix}
        a_{11} & a_{12} \\
        a_{21} & a_{22}
    \end{matrix}\right), \quad
    M_2=\left(\begin{matrix}
        b_{11} & b_{12}\\
        b_{21} & b_{22}
    \end{matrix}\right),
\end{equation}
where $a_{ij}, b_{ij} \in \mathbb{R}$, $a_{11}a_{22}-a_{12}a_{21}=0$ and $(b_{11}-b_{22})^2 +4b_{12}b_{21} < 0$.
\end{system}

Let $\lambda_3$ and $\overline{\lambda}_3$ denote two eigenvalues of $M_2$.
The matrix $M_2$ is similar to its real Jordan form (see for example Theorem 3.4.1.5. in \cite{horn2012matrix}),
\begin{equation*}
    J=|\lambda_3| \left(\begin{matrix}
        \cos \alpha \pi & \sin \alpha \pi \\
        -\sin \alpha \pi &  \cos \alpha \pi
    \end{matrix}\right)
\end{equation*}
where $\cos \rev{\alpha \pi} = \rev{|\lambda_3 + \overline{\lambda}_3|}/2|\lambda_3|$. Thus, there exists an invertible matrix $P$ such that $M_2=PJP^{-1}$. 
Let $M'_1=P^{-1}M_1P$, then the system \ref{sys1} is stabilizable if and only if the following switched system is stabilizable.
\begin{system}\label{sys2}
Consider a two dimensional switched system with a set of matrices $\mathcal{M}=\{M'_1, J\}$ such that
\begin{equation}
    M'_1=\left(\begin{matrix}
        a'_{11} & a'_{12} \\
        a'_{21} & a'_{22}
    \end{matrix}\right), \quad
    J=|\lambda_3|\left(\begin{matrix}
        \cos \alpha \pi & \sin \alpha \pi\\
        -\sin \alpha \pi & \cos \alpha \pi
    \end{matrix}\right),
\end{equation}
where $a_{ij} \in \mathbb{R}$, $a'_{11}a'_{22}-a'_{12}a'_{21}=0$.
\end{system}
\rev{Let $\lambda_1$ and $\lambda_2$, $v_1$ and $v_2$ denote two eigenvalues and \rev{normalized} eigenvectors of $M'_1$. If $\lambda_1=\lambda_2=0$, $M_1$ is nilpotent, then System \ref{sys2} becomes trivial.
Let $\lambda_2 \in \mathbb{R}\setminus\{0\}$ and $\lambda_1 = 0$, then we have following results.}
\begin{lemma}\label{lm6}
The stabilizability radius of System \ref{sys2} is given by
\begin{equation}
        \tilde{\rho} = |\lambda_3| \min \left\{1, \, \inf_{l \in \mathbb{N}} \left| \frac{\lambda_2}{\lambda_3} \frac{\sin(l\alpha - \beta)\pi}{\sin \beta \pi} \right|^{1/(l+1)} \right\},
\end{equation}
where $\beta \pi=\arccos|\left<v_1,v_2\right>|$.
\end{lemma}
\begin{remark}
    \rev{we have $\alpha \in (0,1)$ and $\beta \in (0, 1/2)$.}
\end{remark}
\begin{proof}
    To reduce the norm of $x$, the general approach is to rotate (apply $J$) as close to $v_1$ as possible, then apply $M'_1$. 
    Consider the switching law in following general forms: $A_{\sigma,t}=J^t$ or
    \begin{equation}\label{eq24}
        A_{\sigma,t} = J^{t_k} M'_1 J^{t_{k-1}} \cdots M'_1 J^{t_1},
    \end{equation}
    where $t_i \in \mathbb{N}^{+}$ for $i \in {2, \dots, k-1}$, $t_k,\; t_1 \in \mathbb{N}$ and  
    $$\sum_{i=1}^{k} t_i +k-1= t.$$
    \rev{When $A_{\sigma,t}=J^t$, we have $\tilde{\rho} = |\lambda_3|$. Now consider the optimal sequence of the form \eqref{eq24}.} For any vector $x \in \mathbb{R}^2$, consider the polar coordinates of $x$, i.e. $x = (r,\; \theta \pi)$, we have 
    \begin{equation}
        J^t x= (|\lambda_3|^t r,\; (\theta -t\alpha) \pi).        
    \end{equation}
    Without losing generality, we assume \rev{the line in the direction of $v_1$} is the $x$ axis, Then we have
    \begin{equation}
        M'_1 x= \left(\left|\lambda_2 r \frac{\sin \theta \pi}{\sin \beta \pi}\right|,\; \beta \pi +k\pi\right),\; k \in \mathbb{Z}.
    \end{equation}
    Notice that after the first time applying $M'_1$, the trajectory starts from $v_2$.
    Let $t_i =l$ for $i= 2 \dots k-1$ and $\|M'_1J^{t_1}x\| = C \|x\|$ \rev{for some constant $C$ only depending on $t_1$ and $x$}, and we have
    \begin{equation}\label{eq7}
        \|x_{\sigma,x_0}(t) \| = \left|\lambda_2 |\lambda_3|^l \frac{\sin (\beta - l\alpha) \pi}{\sin \beta \pi}\right|^{k-2} C \|x_0\|. 
    \end{equation}
    \rev{Since $C$ is independent from $k$, from the definition of $\tilde{\rho}$ we have
    \begin{equation}
        \tilde{\rho}=\sup_{x_0} \inf_l \left|\lambda_2 |\lambda_3|^l \frac{\sin (\beta - l\alpha) \pi}{\sin \beta \pi}\right|^{1/(l+1)} = \inf_l \left|\lambda_2 |\lambda_3|^l \frac{\sin (\beta - l\alpha) \pi}{\sin \beta \pi}\right|^{1/(l+1)}.
    \end{equation}}
\end{proof}

To find the infimum in Lemma \ref{lm6}, we need to know $\|l \alpha - \beta \|_N$ for $l \in \mathbb{N}$ with $\| \cdot \|_N$ from number theory denoting the distance to the closest integer.
\rev{This is a Diophantine approximation problem and we need the continued fraction of $\alpha$ to solve it(see for example \cite{rockett1992continued}).} 
Let $[a_0; a_1, a_2, \dots ]$ be the continued fraction of $\alpha$,
\begin{equation}
    \frac{p_k}{q_k} = [a_0;a_1, \dots , a_k] \quad (k \geq 0),
\end{equation}
and $D_k = q_k \alpha - p_k$. 
\rev{Knowing these, let us introduce several lemmas from \cite{beresnevich2020sums}.}
\begin{lemma}\label{lm1}
    (\cite{beresnevich2020sums} Lemma (3.1))  For any $\alpha \in \mathbb{R} \setminus \mathbb{Q}$, we have $l \in \mathbb{N}$ satisfies
    $$q_K \leq l < q_{K+1}$$
    for some $q_K$ of $\alpha$. There exists a unique sequence $\{c_{k+1}\}^{\infty}_{k=0}$ such that
    \begin{equation}
        l= \sum_{k=0}^{\infty} c_{k+1} q_k.
    \end{equation}
    where $c_{k+1}$ satisfies 
    \begin{itemize}
        \item $c_{k+1} = 0, \; \forall k > K$,
        \item $0 \leq c_1 < a_1$ and $0 \leq c_{k+1} \leq a_{k+1}, \forall k \geq 1$,
        \item $c_k =0$ whenever $c_{k+1} =a_{k+1}, \forall k \geq 1$. 
    \end{itemize}
\end{lemma}  
\begin{lemma}\label{lm2}
    (\cite{beresnevich2020sums} Lemma (3.2)) For any $\alpha \in [0,1) \setminus \mathbb{Q}$ and $\theta \in [-\alpha, 1-\alpha)$, there exists a unique integer sequence $\{b_{k+1}\}^{\infty}_{k=0}$ such that
    \begin{equation}
        \theta = \sum_{k=0}^{\infty} b_{k+1} D_k.
    \end{equation}
    where $b_{k+1}$ satisfies 
    \begin{itemize}
        \item $0 \leq b_1 < a_1$ and $0 \leq b_{k+1} \leq a_{k+1}, \forall k \geq 1$,
        \item $b_k =0$ whenever $b_{k+1} =a_{k+1}, \forall k \geq 1$. 
    \end{itemize}
\end{lemma}
Details of $c_{k+1}$ and $b_{k+1}$ can be found in \cite{beresnevich2020sums}.

\begin{lemma}\label{lm3}
    (\cite{beresnevich2020sums} Lemma (4.3)) Consider any $\alpha \in [0,1) \setminus \mathbb{Q}$ and $\theta \in [-\alpha, 1-\alpha)$, and suppose that
    \begin{equation}
        \|l\alpha - \theta \|_N > 0, \; \forall l \in \mathbb{N}.
    \end{equation}
    There exists a smallest integer $k_0$ such that $c_{k_0+1} \neq b_{k_0+1}$ and we define
    \begin{equation}
        \Sigma = \sum_{k=k_0}^{\infty} \delta_{k+1} D_k,
    \end{equation}
    where $\delta_{k+1} = c_{k+1}-b_{k+1}$ with $c_{k+1}$ and $b_{k+1}$ from Lemma \ref{lm1} and \ref{lm2}. Then we have
    \begin{equation}
        \|l\alpha - \theta \|_N = \|\Sigma\|_N.
    \end{equation}
\end{lemma}

\rev{Then we have more information about the infimum in Lemma \ref{lm6}.}
\begin{theorem}\label{thm2}
    For System \ref{sys2} with a given $\beta$, we can obtain the stabilizability radius $\tilde{\rho}$ for the following three cases.
    \begin{enumerate}
        \item For any $\alpha$ such that $\|l \alpha -\beta\|_N =0$ for some $l$, we have $\tilde{\rho}=0$.
        \item For any $\alpha \in (0,1) \cap \mathbb{Q}$ and not in case $1$, let $\alpha = p/q$, then we have
        \begin{equation}
            \tilde{\rho} =|\lambda_3| \min \left\{ 1, \, \min_{l} \left| \frac{\lambda_2}{\lambda_3} \frac{\sin(l\alpha - \beta)\pi}{\sin \beta \pi} \right|^{1/(l+1)} \right\},\; l \in \{0,\dots,q-1\}
        \end{equation}
        \item For any $\alpha \in (0,1) \setminus \mathbb{Q}$ and not in case $1$, there exists a sequence $\{l_n\}^{\infty}_{n=1}$ such that
            \begin{equation}
        \tilde{\rho} = |\lambda_3|\inf_{n} \left| \frac{\lambda_2}{\lambda_3}  \frac{\sin(l_n\alpha - \beta)\pi}{\sin \beta \pi} \right|^{1/(l_n+1)}.
        \end{equation}
        Each $l_n$ is given by
        \begin{equation}
            l_n= \sum_{k=0}^{n} b_{k+1} q_k,
        \end{equation}
        where $b_{k+1}$ is obtained from Lemma \ref{lm2} by letting $\theta = \beta$ when $\alpha \in (0,1-\beta)$, or $ \theta = \beta -1$ when $\alpha \in [1-\beta,1)$.
        \rev{In this case, System \ref{sys2} is stabilizable if $|\lambda_3| \leq 1$.}
    \end{enumerate}
\end{theorem}
\begin{proof}
    \rev{For any $\alpha$ such that $\|l\alpha-\beta\|_N=0$, we have $\sin\|l\alpha-\beta\|_N \pi=0$, then by lemma \ref{lm6} we have $\tilde{\rho}=0$.
    For any rational $\alpha$ and not in case $1$, the set $\{\|l\alpha\|_N | l \in \mathbb{N}\}$ is a finite set, and if we let $\alpha=p/q$, then 
    \begin{equation}
        \|(l+q)\alpha\|_N=\|l\alpha\|_N.
    \end{equation} 
    then by Lemma \ref{lm6}, the infimum of $\tilde{\rho}$ is $|\lambda_3|$ or a minimum and reached by one of $l \in \{0, \dots, q-1\}$.}

    \rev{Now let us prove the case 3. Let $\theta = \beta$ when $\alpha \in (0,1-\beta)$, and $ \theta = \beta -1$ when $\alpha \in [1-\beta,1)$. From Lemma \ref{lm2},  
    we know that $\theta =\sum_{k=0}^{\infty} b_{k+1} D_k$.
    Consider minimizing $\|l\alpha-\theta\|$ by a bounded $l$ first, then we only need $k_0$ in Lemma \ref{lm3} as large as possible.
    By choosing
    \begin{equation}
        l_n = \sum_{k=0}^{n} b_{k+1} q_k,\; n \in \mathbb{N},
    \end{equation}
    we can guarantee that $k_0=n+1$ and $\delta_{k+1} = -b_{k+1}$ for $k \geq k_0$. 
    This means that $l_n\alpha$ is the best approximation of $\beta$ for any $l < l_{n+1}$, i.e. $\|l_n\alpha-\beta\| \leq \|l\alpha-\beta\|$ for $l < l_{n+1}$.
    Now we consider the effect of power term $1/(l_n+1)$. As the set $\{\|l\alpha\|_N\}$ is dense in $(0,1)$, then there exist an $l_0$ such that 
    \begin{equation}\label{eq11}
        \left|\frac{\lambda_2}{\lambda_3}\frac{\sin(l_0\alpha-\beta)\pi}{\sin \beta \pi}\right|<1.
    \end{equation}
    Moreover, for $l_n \geq l_0$ defined above, we always have
    \begin{equation}
        \left|\frac{\lambda_2}{\lambda_3}\frac{\sin(l_n\alpha-\beta)\pi}{\sin \beta \pi}\right|^{1/(l_n+1)}<\left|\frac{\lambda_2}{\lambda_3}\frac{\sin(l\alpha-\beta)\pi}{\sin \beta \pi}\right|^{1/(l+1)},
    \end{equation} 
    for $ l_{n}<l < l_{n+1}$. The infimum is reached by one of the $l_n$ or when $n \rightarrow \infty$.
    We also know $\tilde{\rho} <1$ if $\lambda_3 \leq 1$ from Lemma \ref{lm6} and \eqref{eq11}}
\end{proof}
\begin{remark}
    The sequence $\{l_n\}$ described in Theorem \ref{thm2} is non-decreasing and unbounded, but not necessarily increasing. 
\end{remark}
\begin{remark}
    \rev{The infimum is often reached by a finite $l_n$ in case 3. We will discuss this more in the next subsection.}
\end{remark}

\subsection{\rev{Finiteness property}}
For a given $\beta$, we have the following two cases of $\tilde{\rho}$ for different $\alpha$.  
\begin{enumerate}
    \item The infimum in Lemma \ref{lm6} is reached by a finite $l$.
    \item The infimum in Lemma \ref{lm6} is reached when $n$ tends to infinity.
\end{enumerate}
\rev{In the literature, we have a name for case $1$:} 
\begin{definition}\label{def1}
    A system has finiteness property w.r.t. $\tilde{\rho}$ if there exists a finite matrix sequence $A_{\sigma,t} = M_{\sigma(t)} \cdots M_{\sigma(1)}$ such that
    $\tilde{\rho}^{t}$ is the spectral radius of $A_{\sigma,t}$.
\end{definition}
This finiteness property w.r.t. $\hat{\rho}$ has many known results (\cite{jungers2008finiteness}, for example). \rev{We are interested in the finiteness property w.r.t. $\tilde{\rho}$ for system \ref{sys1}.}

Let $L_\beta(c)$ denote the set of $\alpha$ such that the second case holds and $\tilde{\rho} =c$ for some constant $c$. It is clear that $L_\beta(c)$ only contains irrational $\alpha$, which all belong to the third case in Theorem \ref{thm2}.
\rev{The dimension of some $L_\beta(c)$ are given by the following Theorem.}
\rev{\begin{theorem}\label{lm8}
    The parameter set $L_\beta(c)$ for $0 \leq c < 1$ has zero Hausdorff dimension.
\end{theorem}}
\begin{proof}
    For each $\alpha \in L_\beta(c)$, we have
    \begin{equation}
        \tilde{\rho} =|\lambda_3| \liminf_{n \rightarrow \infty} \left| \frac{\lambda_2}{\lambda_3} \frac{\sin (l_n\alpha-\beta)\pi}{\sin \beta \pi} \right|^{1/(l_n+1)} = c,
    \end{equation}
    with $\{l_n\}$ from the third case in Theorem \ref{thm2}. Let $R_n$ denote $\|l_n\alpha -\beta\|_N$, as $\sin x \sim x$ when $x \rightarrow 0$, there exists a infinite subsequence $\{l_{n_i}\}$ such that
    \begin{equation}
        \lim_{i \rightarrow \infty}  R_{n_i}^{1/(l_{n_i}+1)} =c.
    \end{equation}
    Then there exists a $c <c_1< 1$ such that for any $\alpha$ in $L_\beta(c)$, we have $R_{n_i} = o(c_1^{l_{n_i}})$ when $i$ tends to infinity.
    
    For any $0<s<1$ and  $\epsilon>0$, let $\mathcal{L}(\alpha)$ denote the smallest integer of each $\alpha$ that satisfies
    \begin{itemize}
        \item[(a)] $\mathcal{L}(\alpha) \in \{l_{n_i}\}$,
        \item[(b)] $\|\mathcal{L}(\alpha)\alpha -\beta\|_N < c_1^{\mathcal{L}(\alpha)}$,
        \item[(c)] $\left(t^{1-s}c_1^{st}\right)'|_{t=(\mathcal{L}(\alpha)-1)}<0$
        \item[(d)] $2c_1^{\mathcal{L}(\alpha)}/\mathcal{L}(\alpha) < \epsilon$,
        \item[(e)] $2^s(-s\ln c_1)^{s-2} \Gamma(2-s,-s \ln c_1(\mathcal{L}(\alpha)-1)) < \epsilon$,
    \end{itemize}
    where $\Gamma(a,b)=\int_{b}^{\infty}t^{a-1}e^{-t} \text{d}t$ is the incomplete gamma function. Such $\mathcal{L}(\alpha)$ exists because $\Gamma(a,b) \rightarrow 0$ as $b \rightarrow \infty$.
    Let $\mathcal{L}=\inf \{\mathcal{L}(\alpha) | \alpha \in L_\beta(c)\}$,
    then for any $\alpha \in L_\beta(c)$, by (a) and (b), there exists an $i$ such that $l_{n_i}>\mathcal{L}$ and
    \begin{equation}\label{eq10}
        \|l_{n_i}\alpha -\beta\|_N < c_1^{l_{n_i}}. 
    \end{equation}
    Then there exists an integer $j$ such that
    \begin{equation}
        \left\|\alpha-\frac{\beta + j}{l_{n_i}}\right\|_N < \frac{c_1^{l_{n_i}}}{l_{n_i}},
    \end{equation}
    Then we have
    \begin{equation}\label{eq1}
        \alpha \subseteq \left(\frac{\beta + j}{l_{n_i}}-\frac{c_1^{l_{n_i}}}{l_{n_i}},\frac{\beta + j}{l_{n_i}}+\frac{c_1^{l_{n_i}}}{l_{n_i}}\right).
    \end{equation}
    The $l_{n_i}$ and $j$ in (\ref{eq1}) depend on $\alpha$.
    Let $I(\alpha)$ denote the interval in (\ref{eq1}), then we have a cover of $L_\beta(c)$ given by
    \begin{equation}\label{eq8}
        L_\beta(c) \subseteq \bigcup_{\alpha \in L_\beta(c)} I(\alpha) \subseteq \bigcup_{k=\mathcal{L}}^{\infty} \bigcup_{j=0}^{k-1} \left(\frac{\beta + j}{k}-\frac{c_1^k}{k},\frac{\beta + j}{k}+\frac{c_1^k}{k}\right).
    \end{equation}
    Let $I(k,j)$ be the $j$th interval of $k$ in \ref{eq8}, by (d) we have $|I(k,j)| = 2c_1^k/k < \epsilon$. By (c) and (e), we have 
    \begin{equation}
        \sum_{k=\mathcal{L}}^{\infty}k|I(k,j)|^s < 2^s\int_{\mathcal{L}-1}^{\infty}t^{1-s}c_1^{st} \text{d}t =2^s(-s\ln c_1)^{s-2} \Gamma(2-s,-s \ln c_1 (\mathcal{L}-1)) < \epsilon.
    \end{equation}
    Then for any $0<s<1$, the Hausdorff measure of $L_\beta(c)$ is zero, so $L_\beta(c)$ is a zero Hausdorff dimension set.
\end{proof}
We have a immediate corollary from Theorem \ref{lm8}.
\begin{corollary}\label{thm3}
    For a given $\beta$, $c \in [0,1)$, let $S_\beta(c)$ be the set of $\alpha$ such that $\tilde{\rho}=c$, then $S_\beta(c)$ is a zero Hausdorff dimension set.
\end{corollary}
\begin{proof}
    Since $0 \leq c < 1$, from Lemma \ref{lm6}, we have
    \begin{equation}
        S_\beta(c) = L_\beta(c) \bigcup \left\{ \alpha \, | \, \exists l, \, \|l\alpha-\beta\|_N \pi =\arcsin \left|\frac{c^{l+1} \sin \beta \pi}{\lambda_2 \lambda^l_3}\right| \right\} .
    \end{equation}
    For any $l \in \mathbb{N}$, there are only finite $\alpha$ satisfies that $\|l\alpha-\beta\|_N $ equals to a known constant. Then $S_\beta(c)$ is the union of a countable set and a set with zero Hausdorff dimension, \rev{hence a zero Hausdorff dimensional set}.   
\end{proof}
\rev{We know that System \ref{sys2} has finiteness property w.r.t. $\tilde{\rho}$ for $\alpha \in (0,1) \setminus \cup_{c}L_\beta(c)$. However, This union of $L_\beta(c)$ over $c$ is not necessarily a zero Hausdorff dimensional set, so we need more restriction on System \ref{sys2} to ensure it has zero Hausdorff dimension.}
\rev{\begin{corollary}\label{col2}
    For System \ref{sys1} with fixed $\beta$, $\cup_{c}L_\beta(c) \cap S$ has zero Hausdorff dimension if $|\lambda_3| \leq 1$, where $S$ is denoted by 
    \begin{equation}
        S =\left\{\alpha \in(0,1) \;:\; \inf_{l} \left| \frac{\lambda_2}{\lambda_3} \frac{\sin (l\alpha-\beta)\pi}{\sin \beta \pi} \right|^{1/(l+1)}<c_0<1 \right\},
    \end{equation}
    for an arbitrary constant $c_0 \in (0,1)$.
\end{corollary}}
\begin{proof}
    \rev{As mentioned in the proof of Theorem \ref{thm2}, $\{l\alpha\}$ is dense in $(0,1)$ for irrational $\alpha$, so $S$ is not empty.
    Consider System $1$ with $|\lambda_3|\leq1$. By lemma \ref{lm6}, we know $S$ is equivalent to the set 
    \begin{equation}
        S=\{\alpha \in (0,1)| \tilde{\rho}<c_0<1\}.
    \end{equation} 
    Then for any $\alpha \in \cup_c L_\beta(c) \cap S$, we have $R_{n_i} = o(c_0^{l_{n_i}})$, where $R_{n_i}$ is same as in the proof of Theorem \ref{lm8}.
    Similar to the proof of Theorem \ref{lm6}, for any $0<s<1$ and  $\epsilon>0$, let $\mathcal{L}(\alpha)$ denotes the smallest integer of each $\alpha$ satisfy that
    \begin{itemize}
        \item $\mathcal{L}(\alpha) \in \{l_{n_i}\}$,
        \item $\|\mathcal{L}(\alpha)\alpha -\beta\|_N < c_0^{\mathcal{L}(\alpha)}$,
        \item $\left(t^{1-s}c_0^{st}\right)'|_{t=(\mathcal{L}(\alpha)-1)}<0$
        \item $2c_0^{\mathcal{L}(\alpha)}/\mathcal{L}(\alpha) < \epsilon$,
        \item $2^s(-s\ln c_0)^{s-2} \Gamma(2-s,-s \ln c_0(\mathcal{L}(\alpha)-1)) < \epsilon$,
    \end{itemize}
    and $\mathcal{L}=\inf \{\mathcal{L}(\alpha) | \alpha \in \cup_c L_\beta(c) \cap S\}$.
    Then there exists an $i$ such that $l_{n_i}>\mathcal{L}$ and
    \begin{equation}
        \alpha \in \left(\frac{\beta+j}{l_{n_i}}-\frac{c_0^{l_{n_i}}}{l_{n_i}},\frac{\beta+j}{l_{n_i}}+\frac{c_0^{l_{n_i}}}{l_{n_i}}\right).
    \end{equation}
    We have then constructed a cover of $\cup_c L_\beta(c) \cap S$ given by
    \begin{equation}
        \bigcup_c L_\beta(c) \cap S \subseteq \bigcup_{k=\mathcal{L}}^{\infty}\bigcup^{k-1}_{j=0} \left(\frac{\beta+j}{l_{n_i}}-\frac{c_0^{l_{n_i}}}{l_{n_i}},\frac{\beta+j}{l_{n_i}}+\frac{c_0^{l_{n_i}}}{l_{n_i}}\right),
    \end{equation}
    and same as the proof of Theorem \ref{lm8}, each interval and the union give zero Hausdorff measure for $\cup_c L_\beta(c) \cap S$.}
\end{proof}
\rev{From \ref{col2}, we know that if we only consider System \ref{sys2} with a given $\beta$, $\alpha \in S$ and $|\lambda_3| \leq 1$. Then the subset of $S$, such that System \ref{sys2} does not have finiteness property w.r.t. $\tilde{\rho}$, is a zero Hausdorff dimensional set.}

\rev{For rational $\beta$, we do not need to consider $S$.} 
\begin{proposition}\label{lm9}
    For rational $\beta$ we have $\cup_{c \in [0,1)} L_\beta(c) \subset L$, where $L$ denotes the set of all Liouville numbers in $(0,1)$.  
\end{proposition}
\begin{proof}
    \rev{For every positive integer $j$, a Liouville number $x$ satisfies that there exist integers $p$ and $q$ with $q > 1$ such that 
    \begin{equation}
        0 < \|x - \frac{p}{q}\|_N < \frac{1}{q^j}.
    \end{equation}
    For any $\alpha \in \cup_{c \in [0,1)} L_\beta(c)$, there exist a $c \in [0,1)$ such that $\alpha \in L_\beta(c)$. By Equation~(\ref{eq10}) with $\beta=p_0/q_0$, we have
    \begin{equation}
        \|l_{n_i}\alpha-\frac{p_0}{q_0}\|_N <c_1^{l_{n_i}},
    \end{equation}
    for a real number $c_1$ such that $c<c_1<1$. Then for any $j$, there exists a sufficiently large $i$ such that $(q_0 l_{n_i})^j c_1^{l_{n_i}}  < l_{n_i} $, i.e. 
    \begin{equation}
        \|\alpha-\frac{p}{q_0 l_{n_i}}\|_N = \frac{1}{l_{n_i}}\|l_{n_i} \alpha - \frac{p_0}{q_0} \|_N < \frac{c_1^{l_{n_i}}}{l_{n_i}} < \frac{1}{(q_0 l_{n_i})^j},
    \end{equation}
    for some $p \in \mathbb{Z}$.}
\end{proof}
It is well known that the set of Liouville numbers has zero Hausdorff dimension (see Theorem $2.4$ in \cite{oxtoby2013measure} for example). 
To summarize, for rational $\beta$, the set $\cup_{c \in [0,1)} S_\beta(c)$
contains all $\alpha$ for which System \ref{sys2} is stabilizable. Moreover,
$$\bigcup_{c \in [0,1)} S_\beta(c) \setminus \bigcup_{c \in [0,1)} L_\beta(c)$$
contains all $\alpha$ for which System \ref{sys2} is stabilizable and has the finiteness property with respect to $\tilde{\rho}$, while
$$\bigcup_{c \in [0,1)} L_\beta(c)\subset L$$
is a set of Hausdorff dimension zero.
     
\section{Two dimension: Hyperbolic case}
\subsection{A singular matrix and a hyperbolic matrix}
In section $4$ we considered about systems with $M_2$ having only non real eigenvalues. Now we consider $M_2$ having only real eigenvalues. 
\begin{system}\label{sys3}
Consider a two dimensional switched system with a set of matrices $\mathcal{M}=\{M_1, M_2\}$ with $M_1$ a singular matrix and $M_2$ a arbitrary matrix with real eigenvalues, namely
\begin{equation}\label{eq14}
    M_1=\left(\begin{matrix}
        a_{11} &a_{12} \\
        a_{21} &a_{22}
    \end{matrix}\right), \quad
    M_2=\left(\begin{matrix}
        b_{11} &b_{12}\\
        b_{12} &b_{22}
    \end{matrix}\right),
\end{equation}
where $a_{ij},b_{ij} \in \mathbb{R}$, $a_{11}a_{22}-a_{12}a_{21}=0$ and $(b_{11}-b_{22})^2 +4b_{12}b_{21} \geq 0$.
\end{system}
Similar to System \ref{sys1}, we need the Jordan form of $M_2$ in System \ref{sys3} for the stabilizability radius. Let $\lambda_4$ and $\lambda_5$ be the two eigenvalues of $M_2$. 
We have the following Proposition.
\begin{proposition}
    The Jordan form $J$ of $M_2$ in System \ref{sys3} equals to one of following:
    \begin{equation}
        J_1=\left(\begin{matrix}
        \lambda_4 & 0 \\
        0 &\lambda_4
    \end{matrix}\right), \quad J_2=\left(\begin{matrix}
        \lambda_4 & 1 \\
        0 &\lambda_4
    \end{matrix}\right), \quad J_3=\left(\begin{matrix}
        \lambda_4 & 0 \\
        0 &\lambda_5
        \end{matrix}\right).
    \end{equation}
\end{proposition}
\begin{proof}
    Since $(b_{11}-b_{22})^2 +4b_{12}b_{21} \geq 0$ for $M_2$, we know $\lambda_4$ and $\lambda_5$ are real. If $\lambda_4 =\lambda_5$, we have $J_1$ or $J_2$, otherwise we have $J_3$.
\end{proof}
Similarly, system \ref{sys3} is stabilizable if and only if the system \ref{sys4} is stabilizable. 
\begin{system}\label{sys4}
Consider a two dimensional switched system with a set of matrices $\mathcal{M}=\{M'_1, J\}$ such that
\begin{equation}\label{eq19}
    M'_1=\left(\begin{matrix}
        a'_{11} &a'_{12} \\
        a'_{21} &a'_{22}
    \end{matrix}\right), \quad
    J= \{J_1, J_2, J_3\}
\end{equation}
where $a'_{11}a'_{22}-a'_{12}a'_{21}=0$.
\end{system}
Let $\lambda_1$ and $\lambda_2$, $v_1=(x_1,y_1)^T$ and $v_2=(x_2,y_2)^T$ denote two eigenvalues and unit eigenvectors of $M'_1$. Similar to section $4$, we only consider System \ref{sys4} with $\lambda_2 \in \mathbb{R}\setminus\{0\}$ and $\lambda_1 = 0$.
For each $v_i$, without loss of generality, we let $x$ coordinates of $v_i$ be in [0,1]. For different $J$ we have different results. 

If $J=J_1$, obviously we have $\tilde{\rho}=\min\{|\lambda_2|, |\lambda_4|\}$.

If $J=J_2$, we have the following theorem.
\begin{theorem}\label{thm4}
    Let Case A denote the case when $y_1$ and $y_2$ satisfy the following condition:
    \begin{equation}
        y_1y_2\lambda_4(y_1 - y_2) \geq 0.
    \end{equation}
    The stabilizability radius of System \ref{sys4} is given by
    \begin{enumerate}
        \item $\tilde{\rho}=\min\{|\lambda_2|, |\lambda_4|\}$, if in Case A.
        \item   \begin{equation}\label{eq18}
                    \tilde{\rho} = \min_{l \in \mathbb{N}} \left|\lambda_2\lambda^l_4 \left(1-\frac{y_1y_2}{\lambda_4(x_1y_2-x_2y_1)}l\right)\right|^{1/(l+1)},
                \end{equation}
            if not in Case A.
    \end{enumerate}
    Here $l \in \{L, L+1\}$ for $L=\lfloor \lambda_4(x_1y_2-x_2y_1)/y_1y_2 \rfloor$ and $\lfloor \cdot \rfloor$ denotes the floor function of a real number.
\end{theorem}
\begin{proof}
    When $\lambda_4 =0$, $J^2_3$ is a zero matrix, we have $\tilde{\rho}=0$. Now we consider the case when $\lambda_4 \neq 0$.
    By Proposition \ref{lm7}, we only need to consider the optimal sequence of initial vector $v_2$.
    There are only two types of optimal sequence for $v_2$.
    \begin{enumerate}
        \item The sequence only contains $J_2$, 
        \item $(M'_1 J^l_2)^k$ for some $l \in \mathbb{N}$.
    \end{enumerate}
    For type $1$ optimal sequence, the stabilizability radius is $|\lambda_4|$.
    For type $2$ optimal sequence, we note
    \begin{equation}
        J^l_2 =\lambda^{l-1}_4 \left(\begin{matrix}
            \lambda_4 & l \\
            0 &\lambda_4
        \end{matrix}\right).
    \end{equation}
    Hence $J^l_2v_2=\lambda^{l-1}_4(\lambda_4x_2+ly_2,\lambda_4y_2)^T$.
    Applying $M'_1$ on $J^l_2v_2$ is projecting $J^l_2v_2$ onto $v_2$ along the direction of $v_1$, then extending it by $|\lambda_2|$. By the law of sines, we have that the norm of $M'_1J^l_2v_2$ is given by
    $$ \|M'_1J^l_2v_2\| = \left|\lambda_2 \|J^l_2v_2\| \frac{\sin \gamma(l) }{\sin \gamma(0) }\right|,$$
    where
    $$ \gamma(l)  = \arctan \frac{1}{x_2/y_2 +l /\lambda_3}-\arctan \frac{y_1}{x_1},$$
    and 
    $$\|J^l_2v_2\| =|\lambda_3|^l  |y_2| \sqrt{1+(x_2/y_2 +l/\lambda_3)^2}.$$
    As required, $\lambda_1 \neq \lambda_2$, then $\gamma(0) \neq 0$.
    Let $I(l)=x_2/y_2 +l/\lambda_3$, for $\sin \gamma (l)$, we have
    \begin{equation}
        \begin{aligned}
        \sin \gamma(l) &=\sin \arctan \frac{1}{I(l)} \cos \arctan \frac{y_1}{x_1} - \cos \arctan \frac{1}{I(l)} \sin \arctan \frac{y_1}{x_1}\\
        &=\frac{1}{\sqrt{1+I(l)^2}} \frac{x_1}{\sqrt{x^2_1+y^2_1}} - \frac{I(l)}{\sqrt{1+I(l)^2}} \frac{y_1}{\sqrt{x^2_1+y^2_1}}\\
        &=\frac{x_1 - I(l)y_1}{\sqrt{1+I(l)^2} \sqrt{x^2_1+y^2_1}}.
        \end{aligned}
    \end{equation}
    Since $v_1$ and $v_2$ are unit vectors, we have $\sin \gamma = x_1y_2-y_1x_2$ and the stabilizability radius is given by  
    \begin{equation}
        \tilde{\rho} = \inf_{l \in \mathbb{N}} \left|\lambda_2 \lambda^l_4 y_2 \sqrt{1 + I(l)^2} \frac{\sin \gamma(l)}{\sin \gamma(0)}\right|^{1/(l+1)}=|\lambda_3| \inf_{l \in \mathbb{N}} \left|\frac{\lambda_2}{\lambda_4}  (1- a l)\right|^{1/(l+1)},
    \end{equation}
    where $a = y_1y_2/\lambda_4  (x_1y_2-y_1x_2)$.
    
    Consider the function $f_1(x) = |\lambda_2(1- a x)/\lambda_4|^{1/(x+1)}, \; x >0$. We have following cases:
    \begin{itemize}
        \item If $y_1y_2 = 0$, then $a = 0$, $\tilde{\rho}=|\lambda_4| \inf_l |\lambda_2/\lambda_4|^{1/(l+1)}$. If $|\lambda_2| \leq |\lambda_4|$, let $l=0$ we have $\tilde{\rho}=|\lambda_2|$. 
              If $|\lambda_2| > |\lambda_4|$, let $l \rightarrow \infty$, we have $\tilde{\rho} = |\lambda_4|$. However $\tilde{\rho} = |\lambda_4|$ implies that $M_1$ is redundant for the optimal sequence. Then when $|\lambda_2| > |\lambda_4|$ we have the optimal sequence of type $1$.
        \item If $y_1y_2 < 0$ and $y_1 < 0$ when $\lambda_4 > 0$, $y_1y_2 > 0$ and $y_1x_2 > x_1y_2$ when $\lambda_4 > 0$, $y_1y_2 < 0$ and $y_2 < 0$ when $\lambda_4 < 0$, or $y_1y_2 > 0$ and $y_1x_2 < x_1y_2$ when $\lambda_4 < 0$, we have $f_1(x)=(\lambda_2(1+|a| x)/\lambda_4)^{1/(x+1)}$. 
              Let $y=1/(x+1)$, we have $y \in (0,1]$ and
              \begin{equation}
                \ln f_1(y) = y \ln \left(\left|\frac{\lambda_2}{\lambda_4}\right|\left(1-|a|+\frac{|a|}{y}\right)\right).
              \end{equation} 
              It is easy to check that $\frac{\text{d}^2}{\text{d}y^2}\ln f_1(y) <0$. Thus the minimum of $\ln f_1(y)$ is reached by end points, i.e. $y=1$ or $y \rightarrow 0$.
              Then the infimum in $\tilde{\rho}$ is reached by $l=0$ or $l \rightarrow \infty$. This case coincides with the above case.
        \item For other cases, $f_1(x) = (\lambda_2(1-|a| x)/\lambda_4)^{1/(x+1)}$ has a zero at $x_0 = 1/|a|$. Then the infimum in $\tilde{\rho}$ is attained at $l = \lfloor 1/c_2 \rfloor \text{ or } \lfloor 1/c_2 \rfloor +1$.
    \end{itemize}
    Notice that $x_1 = \sqrt{1-y_1^2}$ and $x_2 = \sqrt{1-y_2^2}$, so the inequality $y_1x_2 >y_2x_1$ is equivalent to a simpler version $y_1 > y_2$.
    Combining the first and the second case, we have $\tilde{\rho}= \min\{|\lambda_2|,|\lambda_4|\}$ when $y_1y_2\lambda_3(y_1 - y_2) = 0$.    
\end{proof}
\begin{remark}
    In Case A, System \ref{sys4} is not $\tilde{\rho}$-minimal.
\end{remark}

If $J=J_3$, without losing generality, we assume $|\lambda_4| > |\lambda_5|$.
To establish the main result, we first provide Lemma \ref{lm14}.
\begin{lemma}\label{lm14}
    Consider a function
    \begin{equation}
    f(x) = \left( \frac{\eta^x+b}{1+b} \zeta \right)^{1/(x+1)}, \; x \geq 0,
    \end{equation}
    where $b>0$, $\zeta >0$ and $0 < \eta <1$.
    This function has a unique finite minimum if and only if $\eta^{1/(b+1)} < \zeta < (b+1)/b$.
    Also let $x_0$ denote this minimum, then $x_0$ is the unique solution of the following equation.
    \begin{equation}\label{eq17}
        \frac{b\eta^{-x_0} +1}{x_0+1} \ln \left( \frac{\eta^{x_0} + b}{1+b} \zeta \right) = \ln \eta.
    \end{equation}
    When $\zeta \geq (b+1)/b$, $f(x)$ is monotonically decreasing. When $\zeta \leq \eta^{1/(b+1)}$, $f(x)$ is monotonically increasing.
\end{lemma}
\begin{proof}
    we have
    \begin{equation}
        f'(x) = \frac{f(x)}{x+1} \left(-\frac{1}{x+1}\ln \left(\frac{b +  \eta^x}{b + 1} \zeta\right) + \frac{\ln \eta}{b\eta^{-x} + 1}\right)
    \end{equation}
    Let $f_2(x) = \ln(\zeta(\eta^x+b)/(1+b))/(x+1), \text{ and } f_3(x)= \ln \eta/(b \eta^{-x}+1)$.
    It's easy to see $f_3(x)$ is monotonically increasing, $f_3(x) \rightarrow \ln \eta/(b+1)$ as $x \rightarrow 0$ and  $f_3(x) \sim \eta^x\ln \eta/c$ as $x \rightarrow \infty$.
    For $f_2(x)$, we have $f_2(x) \rightarrow \ln \zeta$ as $x \rightarrow 0$ and  $f_2(x) \sim \ln(\zeta b/(1+b))/(x+1)$ as $x \rightarrow \infty$.
    
    The derivative of $f_2(x)$ is given by
    \begin{equation}
        f'_2(x)=-\frac{1}{(x+1)^2}\ln \left(\frac{\eta^x+b}{1+b} \zeta\right) + \frac{1}{x+1} \frac{\ln \eta}{b \eta^{-x}+1}=\frac{1}{x+1}(-f_2(x)+f_3(x)),
    \end{equation}
    Then we have following three cases:
    \begin{itemize}
        \item If $\zeta \geq (b+1)/b$, we have $\ln(\zeta b/(1+b)) \geq 0$, hence that $f_2(x)>0>f_3(x),\; x>0$, namely $f'(x) <0, \; x>0$ for this case. 
        \item If $ \eta^{1/(1+b)}<\zeta<(b+1)/b$, by $\ln(\zeta(b+\eta^x)/(b+1)) \rightarrow \ln(\zeta b/(1+b)) <0$ when $x \rightarrow \infty$, there exists a $X$ such that for $x >X$, $f_2(x) <0$.
              Since $f_3(x)$ increases more rapidly than $f_2(x)$ as $x \to \infty$, then there exists an $x_0 > X$ such that $f_2(x_0) = f_3(x_0)$. Thus $f'_1(x)<0,x<x_0$, $f'_1(x)>0,x>x_0$ and $x_0$ is unique. 
              Meanwhile, $f'(x)<0,\; x<x_0$ and $f'(x)>0,\; x>x_0$ implies that $x_0$ is a minimum of $f(x)$.
        \item If $\zeta \leq \eta^{1/(1+b)}$, we have $f_2(0) \leq f_3(0)$. First assume there exists an $x_1 \geq 0$ such that $f_2(x_1) = f_3(x_1)$. In particular, $x_1 = 0$ when $\zeta = \eta^{1/(1+b)}$.
        
              Now suppose there exists an interval $(x_1, x_2)$ such that $f_2(x) > f_3(x)$ for all $x \in (x_1, x_2)$. 
              Since $f_3(x)$ increases more rapidly than $f_2(x)$, the upper bound $x_2$ must be finite and satisfy $f_2(x_2) = f_3(x_2)$. 
              Given that $f_2(x) > f_3(x)$ on this interval, it follows that $f'_2(x)<0$ on this interval. 
              However, if $f_2(x_1) = f_3(x_1)$, $f_2(x_2) = f_3(x_2)$ and $f_2(x)$ is monotonically decreasing on $(x_1, x_2)$, this contradicts the fact that $f_3(x)$ is a monotonically increasing function.
        
              Consequently, we conclude that $f_2(x) \leq f_3(x)$ for all $x > 0$, with equality holding only at $x = x_1$ (if it exists). This further implies that $f'(x) \geq 0$ for $x > 0$, where $f'(x) = 0$ for at most one value of $x$.
    \end{itemize}
\end{proof}
With Lemma \ref{lm14} we can prove our main result.
\begin{theorem}\label{thm5}
    Let $b=|y_1x_2/x_1y_2|$,  $\eta = |\lambda_5/\lambda_4|$ and $\zeta =|\lambda_2/\lambda_4|$.
    Let Case B denotes the case when $y_1$ and $y_2$ satisfy one of the following conditions:
    \begin{itemize}
        \item $y_2x_1 = 0$,
        \item $y_1y_2 > 0$ and $|y_2| < |y_1|$,
        \item $y_1y_2 <0$ and $\zeta \in \left(-\infty,\; \ln \eta/(b+1) \right] \cup \left[(b+1)/b,\; \infty \right)$.
    \end{itemize}
    The stabilizability radius of system \ref{sys4} is given by
    \begin{enumerate}
        \item $\tilde{\rho}= \min \{|\lambda_2|, |\lambda_4|\}$, if in Case B.
        \item $\tilde{\rho}= \min \{|\lambda_2|, |\lambda_5|\}$, if not in Case B and $y_1x_2 = 0$.
        \item If not in case B and $y_1x_2 \neq 0$, we have  
        \begin{equation}
            \tilde{\rho} = \min_l \left|\lambda_2 \frac{|\lambda_5|^l y_2 x_1 - |\lambda_4|^l y_1 x_2}{y_2x_1 - y_1 x_2}\right|^{1/(l+1)},
        \end{equation}
        where $l \in \{L, L+1\}$ for $L=\lfloor x_0 \rfloor$, and $x_0$ satisfies
        \begin{equation}
            L=\left\{\begin{aligned}
                &\frac{b\eta^{-x_0} +1}{x_0+1} \ln \left( \frac{\eta^{x_0} + b}{1+b} \zeta \right) = \ln \eta, &&\text{ if } y_1y_2 < 0 \text{ and } \eta^{1/(1+b)} < \zeta < (1+b)/b\\
                &x_0 = \frac{\ln b}{\ln \eta}, &&\text{ if } y_1y_2 >0 \text{ and } |y_2| >|y_1|.
            \end{aligned}\right.
        \end{equation}    
    \end{enumerate}
\end{theorem}
\begin{proof}
    Similarly to the proof of Theorem \ref{thm4}, we have the optimal sequence starting from $v_2$ is one of the following two types:
    \begin{enumerate}
        \item The sequence only contains $J_3$.
        \item  $(M'_1 J^l_3)^k$ for some $l \in \mathbb{N}$.
    \end{enumerate}
    Consider the type $2$ optimal sequence. We have $J^l_3v_2=(\lambda^l_4x_2,\lambda^l_5y_2)^T$ and the norm $\|M'_1J^l_3v_2\|$ is given by
    $$ \|M'_1J^l_3v_2\| = \left|\lambda_2 \|J^l_3v_2\| \frac{\sin \gamma(l)}{\sin \gamma}\right|,$$
    Using coordinate information we have
    $$ \gamma(l) = \arctan \frac{\lambda^l_5y_2}{\lambda^l_4x_2} - \arctan \frac{y_1}{x_1},\;\gamma =\arctan \frac{y_2}{x_2}-\arctan \frac{y_1}{x_1}, $$
    and
    $$\|J^l_3v_2\| =\sqrt{\lambda^{2l}_5y^2_2 + \lambda^{2l}_4x^2_2}.$$
    Then we have
    \begin{equation}\label{eq15}
    \tilde{\rho} =\inf_{l \in \mathbb{N}} \left|\lambda_2 \frac{|\lambda_5|^l y_2 x_1 - |\lambda_4|^l y_1 x_2}{y_2x_1 - y_1 x_2}\right|^{1/(l+1)} = |\lambda_4|\inf_{l \in \mathbb{N}} \left|\frac{\lambda_2}{\lambda_4} \frac{|\lambda_5/\lambda_4|^l y_2 x_1 -  y_1 x_2}{y_2x_1 - y_1 x_2}\right|^{1/(l+1)}.
    \end{equation}
    From \eqref{eq2} we have two obvious cases.
    \begin{itemize}
        \item If $y_2x_1=0$, when $|\lambda_2| \leq |\lambda_4|$, we have $\tilde{\rho} = |\lambda_2|$. Otherwise we have optimal sequence of type $1$, $\tilde{\rho} = |\lambda_4|$.
        \item If $y_1x_2=0$, when $|\lambda_2| \leq |\lambda_5|$, we have $\tilde{\rho} = |\lambda_2|$. Otherwise by letting $l \rightarrow \infty$, we have $\tilde{\rho} = |\lambda_5|$.
    \end{itemize}
        Now let $\eta = |\lambda_5/\lambda_4| \in (0,1)$, $b=|y_1x_2/x_1y_2|$, and $f(x)=\left|\lambda_2(\eta^x - b)/\lambda_4(1 - b)\right|^{1/(x+1)}$. Consider the following cases:
    \begin{itemize}
        \item If $y_1y_2 <0$, for $x>0$, we have
        \begin{equation}
            f(x) = \left(\frac{\eta^x + b }{ 1+ b }\left|\frac{\lambda_2}{\lambda_4}\right|\right)^{1/(x+1)}.
        \end{equation}
        Let $\zeta=|\lambda_2/\lambda_4|$, by lemma \ref{lm3}, we have following different cases of $\tilde{\rho}$.
        \begin{itemize}
            \item[i] When $\zeta \geq (b+1)/b$, we have $f'(x) <0$, $\tilde{\rho}=\min\{|\lambda_2|,|\lambda_4|\}$.
            \item[ii] When $\eta^{1/(b+1)} < \zeta < (b+1)/b$, for $x>0$, $f(x)$ has a minimum $x_0$ satisfying \eqref{eq17}. 
            \item[iii] When $\zeta \leq  \eta^{1/(b+1)}$, $f'(x)>0$. Then $\tilde{\rho}=\min\{|\lambda_2|,|\lambda_4|\}$.
        \end{itemize}
        Then in case ii, the infimum of $\tilde{\rho}$ is attained at $l=\lfloor x_0 \rfloor$ or $l=\lfloor x_0 \rfloor +1$.
        \item If $y_1y_2 > 0$ and $|y_2x_1| <|y_1x_2|$, we have
        \begin{equation}
            \left|\frac{\lambda_5}{\lambda_4}\right|^x |y_2x_1| <  |y_1x_2|, \; \forall x >0.
        \end{equation}
        Using the same notation defined in the previous case, we have $b >1$ for this case. For $x>0$ we have
        \begin{equation}
            f(x) = \left(\frac{b - \eta^x}{b - 1}\right)^{1/(x+1)},
        \end{equation}
        and
        \begin{equation}
            f'(x) = \frac{f(x)}{x+1} \left(-\frac{1}{x+1}\ln \frac{b -  \eta^x}{b - 1} -\frac{\ln \eta}{b\eta^{-x} - 1}\right) >0.
        \end{equation} 
        Then $f(x)$ is monotonically increasing. When $|\lambda_2| \leq |\lambda_4|$, we have $\tilde{\rho} = |\lambda_2|$. Otherwise, we have a type $1$ optimal sequence and $\tilde{\rho} = |\lambda_4|$.
        \item If $y_1y_2 >0$ and $|y_2x_1| >|y_1x_2|$, we have a zero of $f(x)$ at
        $$x'_0 = \frac{\ln (y_1x_2 / y_2x_1)}{\ln |\lambda_5 / \lambda_4|}.$$
        Then the infimum of $\tilde{\rho}$ is attained at $l=\lfloor x'_0 \rfloor$ or $l=\lfloor x'_0 \rfloor+1$.
    \end{itemize}
    Similar to the proof of Lemma \ref{lm1}, we can rewrite $|y_2x_1| < |y_1x_2|$ and $|y_2x_1| < |y_1x_2|$ as $|y_2| < |y_1|$ and $|y_2| > |y_1|$.
\end{proof}
\begin{remark}
    In Case B, System \ref{sys4} is not $\tilde{\rho}$-minimal. It's also worth mentioning that when $y_1x_2 = 0$, System \ref{sys4} could be $\tilde{\rho}$-minimal, see the remark below.
\end{remark}
\begin{remark}
    When $\tilde{\rho} = |\lambda_5|$, The optimal sequence is $M'_1J^\infty_3$. This is the only case when we have a infinitely long aperiodic optimal sequence for System \ref{sys4}.
\end{remark}
\begin{remark}
    The $\eta^{1/(1+b)} <\zeta$ is equivalent to $|\lambda_2| > (|\lambda_5||\lambda_4|^b)^{1/(1+b)}$.
\end{remark}

\begin{remark}(Finiteness property):
    
\rev{In this section, there is not a obvious parameter similar to the $\alpha$ in the elliptic case. The stabilizability radius depends on eigenvalues and eigenvectors. 
By Theorem \ref{thm4} and Theorem \ref{thm5}, System \ref{sys4} has the finiteness property except one case, when $M_2$ has the Jordan form $J_3$, not in Case B, $y_1x_2=0$ and $|\lambda_5| < |\lambda_2|$.}
\end{remark}

\section{More low dimensional examples}
It is helpful to consider a special case of System \ref{sys1} by letting $a=2$ and $b=c=d=0$, i.e the following example.
\begin{example}\label{ex7}
   Consider the switched system with the matrix set $\mathcal{M} = \{M_1, M_2\}$ such that
\begin{equation}\label{eq3}
    M_1 = \left(\begin{matrix}
        2 &0 \\
        0 &0
    \end{matrix}\right), \quad
    M_2 = \left(\begin{matrix}
        \cos \alpha \pi &\sin \alpha \pi\\
        -\sin \alpha \pi &\quad \cos \alpha \pi
    \end{matrix}\right).
\end{equation}
\end{example}
Applying Lemma \ref{lm6} with $\beta = 1/2$, we have
\begin{corollary}\label{col1}
    The stabilizability radius of Example \ref{ex7} is given by
    \begin{equation}
        \tilde{\rho} = \inf_{l \in \mathbb{N}} \left| 2 \cos l\alpha\pi \right|^{1/(l+1)}.
    \end{equation}
\end{corollary}
\begin{figure}
    \centering
    \includegraphics[width=0.6\textwidth]{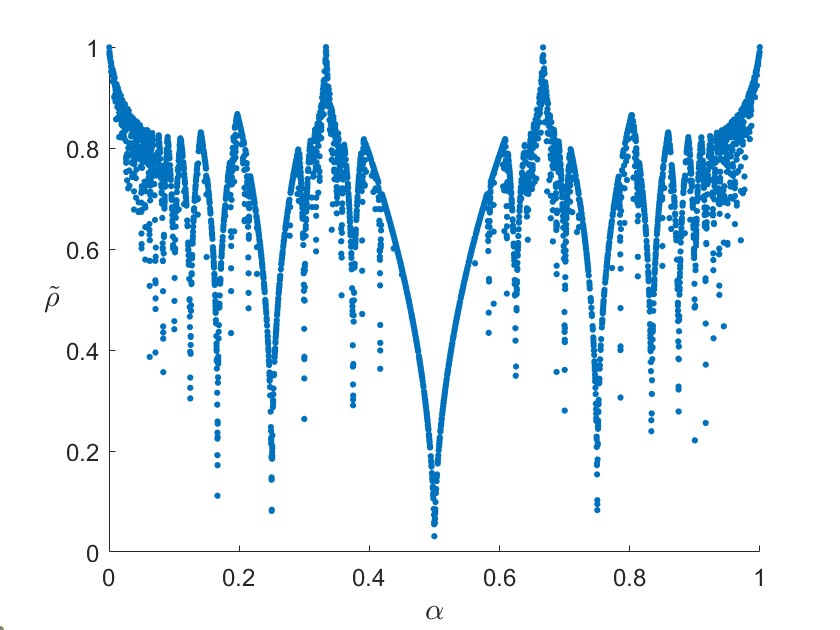}
    \caption{Stabilizability radius for $\alpha \in (0,1)$ \rev{of Example \ref{ex7}}}\label{fig2}
\end{figure}

One property of Example \ref{ex7} is that this system is pointwise stabilizable for most $\alpha$. 
Figure \ref{fig2} shows the $\tilde{\rho}$ in terms of $\alpha$ for Example \ref{ex7}. 
To generate this figure, choose $\alpha$ randomly sampled from $U(0,1)$ and calculate $\tilde{\rho}$ from Corollary \ref{col1}. If only rational $\alpha$ are chosen, all $\tilde{\rho}$ values would be zero.
\begin{proposition}\label{lm5}
    For any $\alpha \in (0,1)$, we have $\tilde{\rho} \leq 1$ and the equality only holds when $\alpha = 1/3 \text{ or } 2/3$.
\end{proposition}
\begin{proof}
    The Diophantine approximation of Example \ref{ex7} is given by $R(\alpha) = \|l\alpha - 1/2\|_N $.
    For irrational $\alpha$, $R(\alpha)$ could be arbitrarily small, so there exists an $l_0$ such that $\|l_0\alpha - 1/2\|_N < 1/6$.

    For rational $\alpha$, it is known that the smallest nontrivial coprime integer of $2$ is $3$, so we have $R(\alpha) \leq R(1/3) = R(2/3) $, i.e. there exists $l$ such that
    \begin{equation}
        \|l\alpha - \frac{1}{2}\|_N \leq \|\frac{l}{3} - \frac{1}{2}\|_N = \frac{1}{6}, \; \forall \text{ rational } \alpha \in (0,1).
    \end{equation}
    Combining the two cases, we have $R(\alpha) \leq 1/6$ for any $\alpha \in (0,1)$.
    Then we have 
    \begin{equation}
        \tilde{\rho} \leq \left|2 \sin \frac{\pi}{6} \right| = 1. 
    \end{equation}
    The equality holds if and only if $R(\alpha) = 1/6$, i.e. $\alpha = 1/3$ or $2/3$.
\end{proof}
\rev{In this example, subset $S$ in Corollary \ref{col2} has Lebesgue measure close to $1$ if $c_0$ is close to $1$.}

We know that when $\alpha = 1/3$, Example \ref{ex7} is unstabilizable.
However, we can obtain a stabilizable system from Example \ref{ex7} by increasing the smallest singular value.
\begin{example}\label{ex4}
    Switched system with $\mathcal{M}=\{M_1, M_2\}$ similar to an example in \cite{stanford1979stability} such that
    \begin{equation}\label{eq6}
        M_1=\left(\begin{matrix}
            2 &0 \\
            0 &1/2    \end{matrix}\right), \quad
        M_2 = \left(\begin{matrix}
            \cos \pi/3 &\sin \pi/3\\
        -\sin \pi/3 &\quad \cos \pi/3
        \end{matrix}\right),
    \end{equation}
\end{example}
\begin{lemma}
    Example \ref{ex4} is pointwise stabilizable.
\end{lemma} 
\begin{proof}
    For any initial vector $x_0 = (\cos \theta, \sin \theta)^T$ with $\theta \in [0, \pi]$, we have one of following norm is less than $1$.
    \begin{itemize} 
        \item $\|A_2A_1x_0\|$.
        \item $\|A_2A_1A_1x_0\|$.
        \item $\|A_2A_1A_2x_0\|$.
        \item $\|A_2A_1A_2A_1x_0\|$.
        \item $\|A_2A_1A_1A_2x_0\|$.
    \end{itemize}
    Figure \ref{fig1} shows the range of $\theta$ such that each of the above norms is less than $1$.
    One can obtain the precise range of $\theta$ by solving the inequalities that the norm is less than $1$, which is not shown here.
\begin{figure}[htbp]
    \centering
    \includegraphics[width=0.6\textwidth]{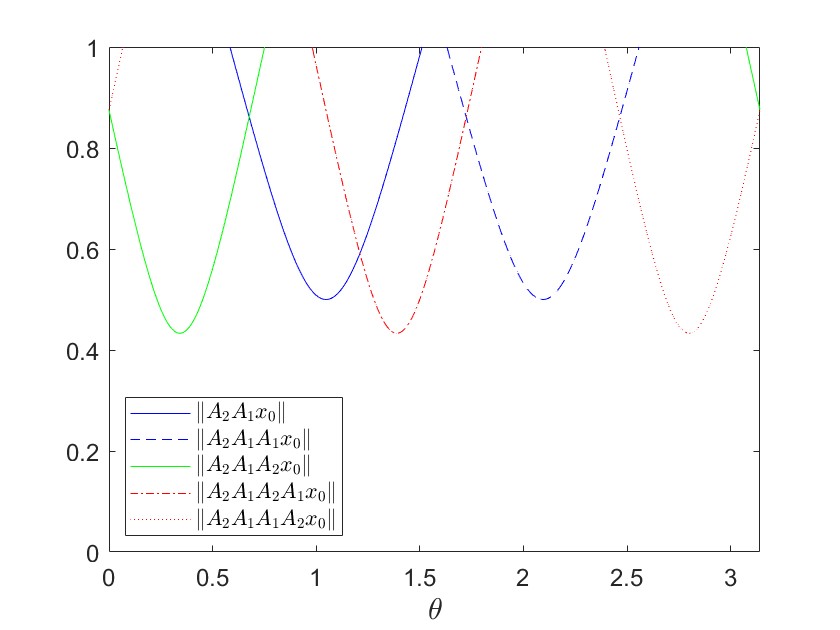}
    \caption{Norm of $x_0$ under various matrix combinations}
    \label{fig1}
    \end{figure}
\end{proof}
Then to make a system stabilizable, we don't necessarily need the smallest singular value as small as possible.

We see that the vector is projected to a lower dimensional subspace after applying a singular matrix. 
One general way of constructing a optimal sequence is to focus on the movement of the image of singular matrices.
For example, after applying $M_2$ once, Example \ref{ex7} only describes movement of the one dimensional image in $\mathbb{R}^2$.
For some singular matrices, we can minimize the dimension of the image first.  
\begin{example}\label{ex3}
system with $\mathcal{M}=\{M_1,M_2\}$ such that
\begin{equation}
    M_1=\left(\begin{matrix}
        1 &0 &0\\
        0 &0 &1\\
        0 &0 &0
    \end{matrix}\right), \quad
    M_2=\left(\begin{matrix}
        \cos\frac{\pi}{5} &0 &\sin\frac{\pi}{5}\\
        0&1 &0\\
        -\sin\frac{\pi}{5}&0 &\cos\frac{\pi}{5}\\
    \end{matrix}\right).
\end{equation}
\end{example}
In Example \ref{ex3}, the image of $M^2_1$ is a one dimensional subspace and $M^2_1$ has the same largest singular value as $M_1$. 
To find the optimal sequence, we can first apply $M_1$ twice to reduce the dimension of the image, then apply $M_2$ to rotate the vector in the image subspace.
Sometimes, matrix combinations can be used to reduce the dimension of the image.
\begin{example}\label{ex1}
    Switched system with $\mathcal{M}=\{M_1,M_2,M_3\}$ such that
\begin{equation}
    M_1=\left(\begin{matrix}
        1 &0 &0\\
        0 &1 &0\\
        0 &0 &0
    \end{matrix}\right), \quad
    M_2=\left(\begin{matrix}
        1 &0 &0\\
        0 &0&-1\\
        0 &1 &0
    \end{matrix}\right), \quad
    M_3=\left(\begin{matrix}
        \cos\frac{\pi}{5} &0 &\sin\frac{\pi}{5}\\
        0&1 &0\\
        -\sin\frac{\pi}{5}&0 &\cos\frac{\pi}{5}\\
    \end{matrix}\right).
\end{equation}
\end{example}
The image of the matrix combination $M_1M_2M_1$ in Example \ref{ex1} is a one dimensional subspace. To stabilize this system, we can apply $M_1M_2M_1$ to any initial vector first.
However, maximizing the dimension of kernel by some matrix sequence is not always the best strategy to find the optimal sequence.
\begin{example}\label{ex2}
    Switched system with $\mathcal{M}=\{M_1,M_2,M_3\}$,
\begin{equation}
    M_1=\left(\begin{matrix}
        2 &0 &0\\
        0 &0 &\frac{1}{2}\\
        0 &0 &0
    \end{matrix}\right), \quad
    M_2=\left(\begin{matrix}
        1 &0 &0\\
        0 &0&-1\\
        0 &1 &0
    \end{matrix}\right), \quad
    M_3=\left(\begin{matrix}
        \cos\frac{\pi}{5} &0 &\sin\frac{\pi}{5}\\
        0&1 &0\\
        -\sin\frac{\pi}{5}&0 &\cos\frac{\pi}{5}\\
    \end{matrix}\right).
\end{equation}
\end{example}
Using computer programs, it's easy to check that the optimal sequence does not start with maximizing the dimension of the kernel, i.e. $M^2_1$.
For example, when $t=10$, the matrix combination $M_3M_2M_1M_3M_2M_3M_1M^2_3M_2$ has smallest spectral radius.

\section{Conclusion}
In this paper, two different approaches to understand the stabilizability radius of linear discrete time switched systems with singular matrices are presented. 
First, we showed that in general, the stabilizability radius is lower bounded by the joint spectral subradius divided by the number of matrices.
\rev{We also show the stabilizability radius equals joint spectral subradius for switched system with rank one matrices.}

\rev{Then we focused on the two dimensional switched system with a singular matrix and an arbitrary matrix, in which we presented a method to compute the exact stabilizability radius. 
For two dimensional switched system with a singular matrix and a matrix with complex eigenvalues, the parameter set is a zero Hausdorff dimension set when the stabilizability radius equals a constant, such
as 0, are presented using this method.}

Although most works of linear discrete time switching systems in the literature focus on general results, we presented some examples to show systems with singular matrices are also interesting. 
The singular matrices function as projections in systems, which could sometimes simplify the optimal sequence of matrices. 
However, it could also lead to more complex behaviors. For instance, in Example \ref{ex7}, the set of $\alpha$ such that the stabilizability radius equals $0$ has zero Hausdorff dimension and is dense in $(0,1)$.
In Example \ref{ex2}, we need to consider more matrix combinations to find the optimal sequence.

There are still many open questions about the stabilizability radius of linear discrete time switching systems, such as finding a more precise lower bound of the stabilizability radius, or giving more situations when the stabilizability radius equals the joint spectral subradius.
In particular, for switched systems with singular matrices, could we find a general method to decide when the stabilizability radius equals zero? Can we say more about the finiteness property w.r.t. the stabilizability radius of System \ref{sys1} or other systems?

\section*{Acknowledgments} The authors are grateful to Alan Haynes for pointing them to Ref.~\cite{beresnevich2020sums}.
For the purpose of open access, the authors has applied a Creative Commons Attribution (CC BY) licence to any Author Accepted Manuscript version arising from this submission.

\bibliographystyle{plain}
\bibliography{refs}
\end{document}